\documentclass{amsart}

\usepackage{latexsym,amssymb}
\usepackage{amsmath,amsthm}

\usepackage{amsfonts}
\usepackage{graphicx}   
\usepackage{url}

\usepackage{mathtools}
\mathtoolsset{showonlyrefs}

\def\dist{\mathop{\text{\normalfont dist}}}

\def\pint{\operatorname {--\!\!\!\!\!\int\!\!\!\!\!--}}

\newcommand{\R}{\mathbb{R}}

\newcommand{\N}{\mathbb{N}}

\newcommand{\ve}{\varepsilon}
\newcommand{\ito}{\infty}

\newcommand{\D}{\mathcal{D}}
\newcommand{\J}{\mathcal{J}}

\newtheorem{thm}{Theorem}[section]
\newtheorem{corol}[thm]{Corollary}
\newtheorem{lemp}[thm]{Lemma}
\newtheorem{prop}[thm]{Proposition}
	
\theoremstyle{definition}
\newtheorem{de}[thm]{Definition}
\theoremstyle{remark}
\newtheorem{rk}[thm]{Remark}

\numberwithin{equation}{section}

\begin{document}

\title[Optimal design problems]{Optimal design problems for the first $p-$fractional eigenvalue 
with mixed boundary conditions}

\author[J. Fern\'andez Bonder, J.D. Rossi and J.F. Spedaletti]{Juli\'an Fern\'andez Bonder, Julio D. Rossi and Juan F. Spedaletti}

\address[J. Fern\'andez Bonder]{Departamento de Matem\'atica FCEN - Universidad de Buenos Aires and IMAS - CONICET. Ciudad Universitaria, Pabell\'on I (C1428EGA)
Av. Cantilo 2160. Buenos Aires, Argentina.}

\email{jfbonder@dm.uba.ar}

\urladdr{http://mate.dm.uba.ar/~jfbonder}

\address[J.D. Rossi]{Departamento de Matem\'atica FCEN - Universidad de Buenos Aires and IMAS - CONICET. Ciudad Universitaria, Pabell\'on I (C1428EGA)
Av. Cantilo 2160. Buenos Aires, Argentina.}

\email{jrossi@dm.uba.ar}

\urladdr{http://mate.dm.uba.ar/~jrossi}

\address[J.F. Spedaletti]{Departamento de Matem\'atica, Universidad Nacional de San Luis and IMASL - CONICET. Ej\'ercito de los Andes 950 (D5700HHW), San Luis, Argentina.}

\email{jfspedaletti@unsl.edu.ar}



\begin{abstract}
In this paper we study an optimal shape design problem for the first eigenvalue of the fractional $p-$laplacian with mixed boundary conditions. The optimization variable is the set where the Dirichlet condition is imposed (that is restricted to have measure equal than a prescribed quantity, $\alpha$). We show existence of an optimal design and analyze the asymptotic behavior when the fractional parameter $s\uparrow 1$ obtaining asymptotic bounds that are independent of $\alpha$.
\end{abstract}

\maketitle

\section{Introduction}
The purpose of this paper is to analyze some optimization problems related to the first eigenvalue of the fractional $p-$laplacian with mixed boundary conditions of Neumann and Dirichlet type, where the optimization variable is the region in which the Dirichlet condition is imposed.

Let us be more specific.  Let $\Omega \subset \R^n$ be an open, bounded and Lipschitz domain. 
The complement of $\Omega$ is then divided into two sets
$$
\R^n\setminus \Omega = N\cup D, \quad N\cap D=\emptyset,
$$
and , for $0<s<1<p<\ito$,
we consider the fractional $p-$laplacian operator of order $s$ with homogeneous Neumann condition on $N$ and homogeneous Dirichlet condition on $D$. That is
$$
(-\Delta_p)^s u(x) = 2K\ \text{p.v.} \int_{\Omega\cup D} \frac{|u(x)-u(y)|^{p-2}(u(x)-u(y))}{|x-y|^{n+sp}}\, dy,
$$
where $K=K(n,s,p)$ is a normalization constant given by
\begin{equation}\label{defK}
K:= (1-s) \left(\pint_{\partial B_1} |e_1\cdot y|^p\, dS_y\right)^{-1} = (1-s) \frac{\sqrt{\pi}\Gamma(\frac{n+p}{2})}{\Gamma(\frac{n}{2})\Gamma(\frac{p+1}{2})},
\end{equation}
here $\Gamma(z) = \int_0^\infty t^{z-1}e^{-t}\, dt$ is the Gamma function.

This operator is thought as acting in the fractional Sobolev space
$$
W^{s,p}_D(\Omega) := \{u\in W^{s,p}(\Omega\cup D)\colon u|_D = 0 \text{ a.e.}\}
$$
where, as usual, the fractional Sobolev space $W^{s,p}(U)$ is defined as 
$$
W^{s,p}(U)=\left \{u\in L^{p}(U)\colon \frac{|u(x)-u(y)|}{|x-y|^{\frac{n}{p}+s}}\in L^{p}(U\times U)\right \}.
$$
Hence, $(-\Delta_p)^s\colon W^{s,p}_D(\Omega)\to [W^{s,p}_D(\Omega)]'$ and
$$
\langle (-\Delta_p)^s u, v\rangle = K \iint_{(\Omega\cup D)\times (\Omega\cup D)} \frac{|u(x)-u(y)|^{p-2}(u(x)-u(y))(v(x)-v(y))}{|x-y|^{n+sp}} dxdy.
$$

With these definitions, the associated eigenvalue problem is
\begin{equation}\label{eigenvalue}
\begin{cases}
(-\Delta_p)^s u = \lambda |u|^{p-2} u & \text{in }\Omega\\
u = 0 & \text{in } D.
\end{cases}
\end{equation}

The set $N$ does not enter in our formulation and is interpreted as a Neumann condition over it, since the system does not interact with $N$.

This problem is analogous in the nonlocal fractional setting to the eigenvalue problem for the
local $p-$Laplacian
$$
\begin{cases}
-\Delta_p u = - \mbox{div} (|\nabla u|^{p-2} \nabla u) = \lambda |u|^{p-2} u & \text{in }\Omega\\
u = 0 & \text{on } \partial \Omega \setminus \Gamma \\
|\nabla u|^{p-2} \partial_\nu u = 0 & \text{on } \Gamma .
\end{cases}
$$
Here $\Gamma\subset \partial \Omega$. When $\Gamma= \emptyset$ we have the usual Dirichlet problem
(with a positive first eigenvalue, $\lambda_1>0$); 
while for $\Gamma=\partial \Omega$ we have the Neumann problem (and here the first eigenvalue is zero, $\lambda_1^N=0$).

Recall that the fractional Sobolev space $W^{s,p}(U)$ is a Banach space when one considers the norm
\begin{equation}\label{normaSobolevclasico}
\|u\|_{W^{s,p}(U)}=\left(\int_U |u|^p\,dx+\iint_{U\times U} \frac{|u(x)-u(y)|^p}{|x-y|^{n+sp}}\,dx \,dy \right)^{1/p}.
\end{equation}
Moreover, for $0<s<1<p<\infty$ it is a reflexive, uniformly convex and separable Banach space.
The term
\begin{equation}\label{seminormaSobolev}
[u]_{W^{s,p}(U)}^p = [u]_{s,p; U}^p= \iint_{U\times U} \frac{|u(x)-u(y)|^p}{|x-y|^{n+sp}}\,dx \,dy,
\end{equation}
is called the {\em Gagliardo seminorm} of $u$. We refer the interested reader to \cite{DiNezza-Palatucci-Valdinoci} for a throughout introduction to these spaces and operators.

It is straightforward to see that $\lambda\in\R$ is an eigenvalue of \eqref{eigenvalue} if and only if is a critical value of the functional
$$
\J(v) := K \iint_{(\Omega\cup D)\times (\Omega\cup D)} \frac{|v(x)-v(y)|^p}{|x-y|^{n+sp}}\, dxdy = K [v]_{s,p; \Omega\cup D}^p,
$$
restricted to the unit ball of $L^p(\Omega)$. Moreover, eigenfunctions of \eqref{eigenvalue} associated to $\lambda$ are critical points of $\J$ on the unit ball of $L^p(\Omega)$ with critical value $\lambda$.

Of particular importance is the {\em first eigenvalue} of \eqref{eigenvalue}, that is given by
$$
\lambda_s(D) = \inf_{v\in W^{s,p}_D(\Omega)} \frac{\J(v)}{ \|v\|_{p;\Omega}^p}.
$$

It is an easy consequence of the direct method of the Calculus of Variations (c.f. Section 2) that
the following assertions hold:
\begin{itemize}
\item the above infimum is achieved, and we can assume that a function $u\in W^{s,p}_D(\Omega)$ that realize the infimum (we will called such a function an {\em extremal}) is normalized in $L^p(\Omega)$, that is, $\|u\|_{p;\Omega}=1$;

\item the number $\lambda_s(D)$ is in fact an eigenvalue of \eqref{eigenvalue};

\item $\lambda_s(D)$ is the first (smallest) eigenvalue, i.e. if $\lambda\in\R$ is an eigenvalue of \eqref{eigenvalue}, then $\lambda\ge \lambda_s(D)$;

\item any eigenfunction associated to $\lambda_s(D)$ has constant sign;

\item $\lambda_s(D)>0$ if $|D|>0$.
\end{itemize} 

The main problem that we address here is the optimization of the {\em principal eigenvalue} $\lambda_s(D)$ with respect to the region $D$ where the Dirichlet data is imposed. With that in mind, we fix a constant $\alpha>0$, define the class
$$
\D_\alpha := \{D\subset \R^n\setminus \Omega\colon \text{measurable and } |D|=\alpha\}
$$
and consider the optimization problems
\begin{equation}\label{lambda.s.+-}
\Lambda_s^+(\alpha) := \sup_{D\in \D_\alpha} \lambda_s(D)\quad \text{and} \quad \Lambda_s^-(\alpha) := \inf_{D\in \D_\alpha} \lambda_s(D).
\end{equation}

As we will see, $\Lambda_s^-(\alpha) = 0$ since pushing $D$ to infinity force that in the limit, so in order to recover a nontrivial constant for the  minimization problem one has to restrict the sets $D$ to be uniformly bounded. Therefore, given $R>0$ large, we define
$$
\D_\alpha^R := \{D\in \D_\alpha\colon D\subset B_R(0)\},
$$
and
\begin{equation}\label{lambda.s.-R}
\Lambda_s^{-,R}(\alpha) := \inf_{D\in \D_\alpha^R} \lambda_s(D).
\end{equation}
This value $\Lambda_s^{-,R}(\alpha)$ is strictly positive.

The main results in this paper are contained in the following theorem.
\begin{thm}\label{teo.main}
Let $\Omega\subset \R^n$ be bounded and open. 
Take $0<s<1<p<\infty$ and fix $\alpha>0$. Let $\Lambda_s^+(\alpha)$, $\Lambda_s^-(\alpha)$, $\Lambda_s^{-,R}(\alpha)$ be given by \eqref{lambda.s.+-} and \eqref{lambda.s.-R}. Then the following hold:
\begin{itemize}
\item $\Lambda_s^-(\alpha)=0$ while $\Lambda_s^{-,R}(\alpha)>0$;

\item $\lim_{s\uparrow 1}\Lambda_s^+(\alpha) = \lambda_1$, where $\lambda_1$ is the first Dirichlet eigenvalue of the local $p-$laplacian in $\Omega$;

\item $\lim_{s\uparrow 1} \Lambda_s^{-,R}(\alpha)= 0$.
\end{itemize}

Moreover, for every $0<s<1$ there exists an optimal set $D_s^R$ for the constant $\Lambda_s^{-,R}(\alpha)$.

Finally, given a sequence of quasi-optimal sets for $\Lambda_s^+(\alpha)$, i.e. $D_s\subset \D_\alpha$ such that
$$
\Lambda_s^+(\alpha) = \lambda_s(D_s) + o(1),
$$
where $o(1)\to 0$ as $s\uparrow 1$, we have that $D_s$ ``surrounds'' the boundary of $\Omega$ in the sense that for any $x\in \partial\Omega$ and any $\ve>0$, there exists $s_0$ such that if $s_0<s<1$, $|D_s\cap B_\ve(x)|>0$.
\end{thm}

Notice that the first Neumann eigenvalue for the $p-$Laplacian in $\Omega$ is $\lambda_1^N =0$ (with
a constant function as eigenfunction). Therefore, as we have that for every set $D$ of measure $\alpha$
in $B_R \setminus \Omega$
$$
\Lambda_s^+(\alpha) \geq \lambda_s (D) \geq \Lambda_s^{-,R}(\alpha),
$$
we conclude that for $s$ close to $1$ we have that $\lambda_s (D)$ 
asymptotically lies between 
the first Dirichlet eigenvalue and the first Neumann eigenvalue for the local $p-$Laplacian in $\Omega$.
In fact, for every $\varepsilon $ there exists $s_0$ such that for $s_0 <s<1$ it holds
$$
 \lambda_1 + \varepsilon \geq \lambda_s (D) \geq  \lambda_1^N =0.
$$
The surprising fact of these asymptotic bounds as $s\uparrow 1$ is that they are independent of the size of the Dirichlet part in our nonlocal problem and they are also independent of the radius of the ball that bounds everything.

To emphasize that in the limit as $s\uparrow 1$ for the lower bound for the eigenvalues we obtain a local problem with Neumann boundary conditions in $\Omega$ we remark that with minor modifications of our arguments we can deal with the eigenvalue problem with a potential $V$. In fact, let us consider $V\in L^\infty(\R^n)$ a potential such that 
\begin{equation}\label{cota.V}
0<v_1\leq V (x) \leq v_2 <+\infty
\end{equation} 
for some numbers $v_1,v_2$ and then we take
$$
\lambda_{s,V}(D) = \inf_{v\in W^{s,p}_D(\Omega)} \frac{\J_V(v)}{ \|v\|_{p;\Omega}^p},
$$
with
$$
\J_V(v) := K \iint_{(\Omega\cup D)\times (\Omega\cup D)} \frac{|v(x)-v(y)|^p}{|x-y|^{n+sp}}\, dxdy 
+ \int_\Omega |u(x)|^p V(x)\, dx .
$$
Associated with this functional we have the optimal constants
\begin{equation}\label{LambdaV}
\Lambda_{s,V}^+(\alpha) := \sup_{D\in \D_\alpha} \lambda_{s,V}(D),\qquad \Lambda_{s,V}^-(\alpha) := \inf_{D\in \D_\alpha} \lambda_{s,V}(D)
\end{equation}
and
\begin{equation}\label{LambdaVR}
\Lambda_{s,V}^{-,R}(\alpha) := \inf_{D\in \D_\alpha^R} \lambda_{s,V}(D).
\end{equation}
In this case, we obtain that
$$
\Lambda_{s,V}^+(\alpha) \to \lambda_1(V)\qquad \mbox{and} \qquad
\Lambda_{s,V}^{-,R}(\alpha) \to \lambda_1^N(V),
$$
as $s\uparrow 1$. Here $\lambda_1 (V)$ and $\lambda_1^N (V)$ are the first eigenvalues for the local $p-$Laplacian with the potential $V$ with Dirichlet and Neumann boundary condition respectively, that are given by
\begin{equation}\label{localVD}
\lambda_1 (V) = \min_{v\in W^{1,p}_0 (\Omega)}  \frac{ \displaystyle
  \int_\Omega |\nabla v|^p\, dx + \int_\Omega |v|^p V(x)\, dx }{\|v\|_{p;\Omega}^p}
\end{equation}
and 
\begin{equation}\label{localVN}
\lambda_1^N (V) = \min_{v\in W^{1,p} (\Omega)}  \frac{ \displaystyle
\int_\Omega |\nabla v|^p\, dx + \int_\Omega |v|^p V(x)\, dx}{\|v\|_{p;\Omega}^p}.
\end{equation}
Notice that in this case $\lambda_1^N (V) \neq 0$.

When we consider a potential $V$ one can also check that  if we don't constraint $D$ into a large ball we still have
$$
\lim_{s\uparrow 1}\Lambda_{s,V}^{-}(\alpha) = \lambda_1^N(V).
$$
In fact, this limit can be deduced taking the limit as $s\uparrow 1$ in the following two
inequalities
$$
\Lambda_{s,V}^{-}(\alpha) \leq \Lambda_{s,V}^{-,R}(\alpha) 
$$
and
$$
\Lambda_{s,V}^{-}(\alpha) \geq \min_{v\in W^{s,p} (\Omega)}
\frac{ \displaystyle  K \iint_{\Omega\times \Omega} \frac{|v(x)-v(y)|^p}{|x-y|^{n+sp}}\, dxdy  + \int_\Omega |v(x)|^p V(x)\, dx}{ \|v\|_{p;\Omega}^p}.
$$
We include some details in Section 5.

With these preliminaries, our second result is the following:
\begin{thm}\label{teo.main2}
Let $\Omega\subset \R^n$ be bounded and open. Let $0<s<1<p<\infty$ be fixed and take $\alpha>0$. Let $V\in L^\infty(\R^n)$ a potential such that \eqref{cota.V} is satisfied.

Let $\Lambda_{s,V}^+(\alpha)$, $\Lambda_{s,V}^-(\alpha)$, $\Lambda_{s,V}^{-,R}(\alpha)$ be the constants defined in \eqref{LambdaV} and \eqref{LambdaVR}. Then the following hold
\begin{itemize}
\item $\lim_{s\uparrow 1}\Lambda_{s,V}^+(\alpha) = \lambda_1(V)$, where $\lambda_1(V)$ is given by \eqref{localVD};

\item $\lim_{s\uparrow 1}\Lambda_{s,V}^{-,R}(\alpha) = \lambda_1^N(V)$, where $\lambda_1^N(V)$ is given by \eqref{localVN}.

\item $\lim_{s\uparrow 1}\Lambda_{s,V}^{-}(\alpha) = \lambda_1^N(V)$.
\end{itemize}
\end{thm}

A very brief comment on related bibliography is in order. Optimal configurations related to eigenvalue problems
is by now a classical subject, 
just to mention a few references we quote \cite{Bucur2, Bucur1, Ch, henrot1, S1, S2}.
On the other hand, nonlocal problems are quite popular nowadays, we just refer to \cite{DiNezza-Palatucci-Valdinoci}
and for references concerning eigenvalues for the nonlocal $p-$Laplacian to \cite{Brasco-Parini-Squassina, Bonder-Spedaletti} and references therein.

\subsection*{Organization of the paper} After this introduction, the rest of the paper is organized as follows: 
In Section 2 we revise some preliminary notions on fractional Sobolev spaces that are needed in the paper. In Section 3 we study the maximization problem and in Section 4 the minimization problem. Finally, in Section 5, we prove Theorem \ref{teo.main2}. Since the proof is similar to the one of Theorem \ref{teo.main} we only sketch it and stress the differences.

\section{Preliminaries}
In this section, we review some definitions on fractional Sobolev spaces and on the $p-$fractional Laplace operator. We believe that most of these results are known to experts and constitute part of the ``folklore'' on the subject, 
but we have chosen to include some proofs of the facts that are needed just for the reader's convenience.

\subsection{A probabilistic interpretation for the mixed boundary conditions}

Let $\Omega\subset \R^n$ and take $N, D\subset \R^n\setminus\Omega$ such that $N\cap D=\emptyset$ and $N\cup D=\R^n\setminus \Omega$.

Suppose that $u=u(x,t)$ measures the density of some substance in space 
$x\in\R^n$ and time $t>0$. Assume that the probability that a particle jumps from a position $x$ to a different position $y$ by time unit is given by a kernel $k = k(x-y)$. Assume moreover that the kernel is symmetric, i.e. $k(z)=k(-z)$.

Now, we further assume that the particles inside $\Omega$ do not interact with the particles in $N$ and inside $D$ the density is zero (every particle that jumps into $D$ is automatically killed). Then, the conservation law for the mass gives rise to the equation
\begin{equation}\label{conservation}
\begin{split}
u_t(x,t) &= \int_\Omega k(x-y) u(y,t)\, dy - \int_{\Omega\cup D} k(x-y) u(x,t)\, dy\\
&= -\int_{\Omega\cup D} k(x-y) (u(x,t)-u(y,t))\, dy.
\end{split}
\end{equation}

In this paper we consider the case where the kernel $k$ behaves like a power of the distance, that is
$$
k(z) \sim |z|^{-(n+2s)},
$$
which means that the particle tends with high probability to stay close to the original position but can jump with positive probability to positions far away.
This type of kernel give rise to the so-called fractional Laplace operator defined as in the introduction,
$$
(-\Delta)^s v(x) = 2 K(n,s,2)\, \text{p.v.}\int_U \frac{v(x)-v(y)}{|x-y|^{n+2s}}\, dy,
$$
where p.v. stands for {\em principal value}.

In some models, the kernel $k$ depends not only on the distance but also on the difference in concentration, making more likely to jump when the difference in concentration is large. i.e.
$$
k = k(x-y, v(x)-v(y)).
$$
Here we consider the case where
$$
k \sim \frac{|v(x)-v(y)|^{p-2}}{|x-y|^{n+sp}},
$$
and these type of kernels give rise to the so-called fractional $p-$laplace operator defined as
\begin{equation}\label{eq.plap}
(-\Delta_p)^s v(x) = 2K(n,s,p) \, \text{p.v.} \int_U \frac{|v(x)-v(y)|^{p-2}(v(x)-v(y))}{|x-y|^{n+sp}}\, dy.
\end{equation}

One key factor to analyze the behavior of the solutions of \eqref{conservation} is the first eigenvalue of the associated operator, i.e. the smallest value of $\lambda\in \R$ such that there exists a nontrivial solution to
$$
\begin{cases}
(-\Delta_p)^s v = \lambda |v|^{p-2} v & \text{in }\Omega\\
v = 0 & \text{in } D
\end{cases}
$$
so the purpose of this paper is to analyze this problem and in particular the optimization of this first eigenvalue with respect to the region $D$ and the asymptotic behavior of this optimal eigenvalue when the fractional parameter $s$ tends to 1.

\subsection{The fractional $p-$laplacian}
In this subsection, we recall some basic facts about the fractional $p-$laplacian given in \eqref{eq.plap}.

\begin{lemp}[\cite{Bonder-Spedaletti}, Lemma 2.2]\label{plap.distribucional}
Let $0<s<1<p<\infty$ be fixed and let $U\subset\R^n$ be open. For every $u\in W^{s,p}(U)$, the fractional $p-$laplacian given by \eqref{eq.plap} defines a distribution $\mathcal D'(U)$. Moreover,
$$
\langle (-\Delta_p)^s u, \phi\rangle = K \iint_{U\times U} \frac{|u(x)-u(y)|^{p-2} (u(x)-u(y))(\phi(x)-\phi(y))}{|x-y|^{n+sp}}\, dx\, dy,
$$
for every $\phi\in C^\infty_c(\Omega)$, where $K$ is defined in \eqref{defK}.
\end{lemp}

\begin{rk}\label{def.lap}
From Lemma \ref{plap.distribucional} one observe that the fractional $p-$laplacian is a bounded operator between $W^{s,p}(U)$ and its dual $[W^{s,p}(U)]'$.
\end{rk}

\begin{rk}\label{rem.ponce}
The choice of the constant $K$ in the definition of the fractional $p-$laplacian is made in order for the operators $(-\Delta_p)^s$ to converge to the local $p-$laplacian $-\Delta_p$ as $s\uparrow 1$. In fact, it is easy to see, from the Gamma convergence results in \cite{Ponce} that for any $f\in L^{p'}(\Omega)$, if $u_s\in W^{s,p}_0(\Omega)$ is the weak solution of 
$$
\begin{cases}
(-\Delta_p)^s u_s = f & \text{in }\Omega\\
u_s = 0 & \text{in }\R^n\setminus \Omega,
\end{cases}
$$
then $u_s\to u$ as $s\uparrow 1$ strongly in $L^p(\Omega)$ where $u\in W^{1,p}_0(\Omega)$ is the weak solution to
$$
\begin{cases}
-\Delta_p u = f & \text{in }\Omega\\
u = 0 & \text{on }\partial \Omega.
\end{cases}
$$
\end{rk}

With all of these preliminaries, we establish the definition of {\em weak solution} for mixed boundary value problem for the fractional $p-$laplacian.
\begin{de}\label{de.sol}
Let $0<s<1<p<\infty$ be fixed and let $\Omega\subset\R^n$ be open and let $D\subset \R^n\setminus \Omega$. Given $f\in L^{p'}(\Omega)$ (or more generally, $f\in [W^{s,p}_D(\Omega)]'$), we say that $u\in W^{s,p}_D(\Omega)$ is a weak solution of 
\begin{equation}\label{eq.sp}
\begin{cases}
(-\Delta_p)^s u = f & \text{in }\Omega,\\
u= 0 & \text{in } D,
\end{cases}
\end{equation}
if the equality holds in the distributional sense. That is, if
$$
K \iint_{(\Omega\cup D)\times (\Omega\cup D)} \frac{|u(x)-u(y)|^{p-2} (u(x)-u(y))(v(x)-v(y))}{|x-y|^{n+sp}}\, dx\, dy = \int_\Omega f v\, dx,
$$
for every $v\in W^{s,p}_D(\Omega)$, where $K$ is given in \eqref{defK}.
\end{de}

The next Poincar\'e-type inequality, although its proof is elementary, is crucial in the remaining of the paper.
This inequality is stablished in \cite[Proposition 2.10]{Bonder-Spedaletti}. We include here a proof for the reader's convenience.
\begin{prop}\label{poincare}
Let $\Omega\subset \R^n$ be a bounded, open set and let $D\subset \R^n\setminus\Omega$ be a measurable bounded set. Then 
$$
[u]_{s,p; \Omega\cup D}^p \ge d^{-(n+sp)}|D| \|u\|_{p; \Omega}^p,
$$
for every $u\in W^{s,p}_D(\Omega)$, where
$$
d=d(\Omega, D) = \sup\{|x-y|\colon x\in\Omega,\ y\in D\}.
$$
\end{prop}

\begin{proof}
The proof is rather simple. In fact, given $u\in W^{s,p}_D(\Omega)$ we have
\begin{align*}
[u]_{s,p;\Omega\cup D}^p &= \iint_{(\Omega\cup D)\times (\Omega\cup D)} \frac{|u(x)-u(y)|^p}{|x-y|^{n+sp}}\, dxdy\\
&\ge \int_{\Omega} |u(x)|^p\left(\int_D \frac{1}{|x-y|^{n+sp}}\, dy\right)\, dx.
\end{align*}
Now, just observe that for every $x\in\Omega$, one has
$$
 \int_D \frac{1}{|x-y|^{n+sp}}\, dy \ge d^{-(n+sp)} |D|, 
$$
and the proof is complete.
\end{proof}

We now want to remove the hypothesis that $D$ is bounded in Proposition \ref{poincare}.
\begin{corol}\label{poincare2} Assume that $D$ is only measurable, then there is a positive constant $\theta >0$ such that
$$
[u]_{s,p; \Omega\cup D}^p \ge \theta \|u\|_{p; \Omega}^p.
$$
\end{corol}

\begin{proof}
Take $R>0$ large so that $D_R := D\cap B_R(0)$ has positive measure. Hence
$$
[u]_{s,p; \Omega\cup D}^p\ge [u]_{s,p; \Omega\cup D_R}^p\ge \theta \|u\|_{p;\Omega}^p; 
$$
where $\theta=\theta_R$ is obtained from Proposition \ref{poincare}. The proof is complete.
\end{proof}

\begin{rk}\label{poincare3}
if $|D|=\alpha<\infty$, one can take $R>0$ in the above proof such that $|D_R|=\frac12 \alpha$. Hence one arrives at
$$
[u]_{s,p; \Omega\cup D}^p\ge d_R^{-(n+sp)}\frac{|D|}{2} \|u\|_{p;\Omega}^p,
$$
where $d_R=\sup\{|x-y|\colon x\in\Omega,\ |y|=R\}$. Therefore one can take
$$
\theta = d_R^{-(n+sp)}\frac{|D|}{2}
$$
for $R$ large.
\end{rk}

From Corollary \ref{poincare2} it is not difficult to show the existence of an extremal for the constant $\lambda_s(D)$. The main difficulty is that, even if $\Omega$ is smooth, since we do not want to make any regularity assumptions on $D$ we cannot assume that the injection $W^{s,p}(\Omega\cup D)\subset L^p(\Omega\cup D)$ is compact.
\begin{thm}\label{existe.extremal}
Let $\Omega\subset\R^n$ be a bounded domain with Lipschitz boundary. Then, given $D\subset \R^n\setminus \Omega$ measurable with positive measure, there exists $u\in W^{s,p}_D(\Omega)$ such that
$$
\lambda_s(D) = \frac{K[u]_{s,p; \Omega\cup D}^p}{\|u\|_{p; \Omega}^p},
$$
where $K$ is given in \eqref{defK}. Moreover, the extremal can be taken to be normalized in $L^p(\Omega)$, i.e. $\|u\|_{p;\Omega} = 1$.
\end{thm}

\begin{proof}
The scheme of the proof is standard. Let $\{v_k\}_{k\in\N}\subset W^{s,p}_D(\Omega)$ be a normalized minimizing sequence for $\lambda_s(D)$, i.e.
$$
\J(v_k) = K [v_k]_{s,p;\Omega\cup D}^p\to \lambda_s(D)\quad \text{and}\quad \|v_k\|_{p;\Omega} = 1.
$$

Then, $\{v_k\}_{k\in \N}\subset W^{s,p}(\Omega\cup D)$ is bounded and hence, up to some subsequence, we can assume that $v_k\rightharpoonup u$ weakly in $W^{s,p}(\Omega\cup D)$. 

Since $v_k=0$ in $D$ for every $k\in\N$ implies that $u=0$ in $D$, so $u\in W^{s,p}_D(\Omega)$.

Moreover, $v_k\rightharpoonup u$ weakly in $W^{s,p}(\Omega)$ and since $\Omega$ is Lipschitz, this implies that the injection $W^{s,p}(\Omega)\subset L^p(\Omega)$ is compact (see \cite[Theorem 7.1]{DiNezza-Palatucci-Valdinoci}) and so $\|u\|_{p;\Omega} = 1$.

Therefore
$$
\lambda_s(D)\le \J(u)\le \liminf_{k\to\infty} \J(v_k) = \lambda_s(D).
$$

The proof is complete.
\end{proof}

\begin{rk}
The Lipschitz regularity on $\partial\Omega$ is needed in order for the compactness of the embedding  $W^{s,p}(\Omega)\subset L^p(\Omega)$ to hold. See  \cite[Theorem 7.1]{DiNezza-Palatucci-Valdinoci}. In fact, what is needed is that $\Omega$ be a {\em bounded extension domain}. That is the existence of a bounded extension operator $E\colon W^{s,p}(\Omega)\to W^{s,p}(\R^n)$. Lipschitz boundary imply that $\Omega$ is a bounded extension domain. See  \cite[Theorem 5.4]{DiNezza-Palatucci-Valdinoci}.
\end{rk}

\section{The maximization problem}

In this section we study the problem of maximization of $\lambda_s(D)$. That is,  given $\alpha>0$ we define
$$
\Lambda_s^+(\alpha) = \sup\{\lambda_s(D)\colon D\subset \R^n\setminus\Omega,\ |D|=\alpha\}.
$$
In particular, we are interested in the asymptotic behavior of the constant $\Lambda_s^+(\alpha)$ when $s\uparrow 1$.
As we will see, the behavior of $\Lambda_s^+(\alpha)$ when $s\uparrow 1$ is independent of the value of the constant $\alpha>0$. 

We begin with a simple lemma.
\begin{lemp}\label{orden}
Let $D_1\subset D_2\subset \R^n\setminus\Omega$ be measurable sets with positive measure. Then $\lambda_s(D_1)\le \lambda_s(D_2)$.
\end{lemp}

\begin{proof}
Let $u\in W^{s,p}_{D_2}(\Omega)$ be such that $\|u\|_{p;\Omega}=1$. Then, $u\in W^{s,p}_{D_1}(\Omega)$ and so
$$
\lambda_s(D_1)\le K [u]_{s,p; \Omega\cup D_1}^p \le K [u]_{s,p; \Omega\cup D_2}^p.
$$
taking infimum on $u\in W^{s,p}_{D_2}(\Omega)$ the conclusion follows.
\end{proof}

As a consequence of this Lemma we have the following upper bound for $\Lambda_s^+(\alpha)$
$$
\Lambda_s^+(\alpha) \le \lambda_s(\R^n\setminus \Omega).
$$
Observe that $\lambda_s(\R^n\setminus\Omega)$ is the first eigenvalue of the fractional $p-$laplacian with Dirichlet Boundary conditions. i.e. is the first eigenvalue of the problem
\begin{equation}\label{frac.dir.eig}
\begin{cases}
(-\Delta_p)^s u = \lambda |u|^{p-2} u & \text{in }\Omega\\
u = 0 & \text{in } \R^n\setminus\Omega.
\end{cases}
\end{equation}

Problem \eqref{frac.dir.eig} was studied, among others, in \cite{Brasco-Parini-Squassina} where, based on the famous results of Bourgain-Brezis-Mironescu \cite{Bourgain-Brezis-Mironescu}, the asymptotic behavior of the eigenvalues when $s\uparrow 1$ is obtained. Namely,
$$
\lim_{s\uparrow 1} \lambda_s(\R^n\setminus\Omega) = \lambda_1,
$$
where $\lambda_1$ is the first eigenvalue of the local $p-$laplacian in $\Omega$ with homogeneous Dirichlet boundary conditions, i.e.
$$
\lambda_1 = \inf_{v\in W^{1,p}_0(\Omega)}\frac{\displaystyle\int_\Omega |\nabla v|^p\, dx}{\displaystyle\int_\Omega |v|^p\, dx}.
$$

On the other hand, take $D= \Omega_r\setminus \overline{\Omega}$ where $\Omega_r = \bigcup_{x\in\Omega} B_r(x)$ is the usual fattening of $\Omega$ and $r$ is chosen in such a way as $|D|=\alpha$.

Then
$$
\lambda_s(D)\le \Lambda^+_s(\alpha), 
$$
and so
$$
\liminf_{s\uparrow 1} \lambda_s(D)\le \liminf_{s\uparrow 1} \Lambda^+_s(\alpha)\le \limsup_{s\uparrow 1} \Lambda^+_s(\alpha) \le \lambda_1(\Omega).
$$

Now, Let $u_s\in W^{s,p}_D(\Omega)$ be the eigenfunction associated to $\lambda_s(D)$ normalized in $L^p(\Omega)$. Then, by \cite[Theorem 4]{Bourgain-Brezis-Mironescu} it follows that there exist a sequence $s_k\uparrow 1$ and a function $u\in W^{1,p}_0(\Omega)$, such that $u_{s_k}\to u$ in $L^p(\Omega)$ and
$$
\|\nabla u\|_{p;\Omega}^p\le \liminf_{k\to\infty} K(n,s_k,p) [u_{s_k}]^p_{s_k,p; \Omega\cup D},
$$
where $K(n,s,p)$ is given in \eqref{defK}. Hence
$$
\|\nabla u\|_{p;\Omega}^p\le \lambda_1(\Omega)
$$
with $\|u\|_{p,\Omega}=1$.
Therefore, $u$ is the normalized eigenfunction of the $p-$laplacian in $\Omega$ and one can conclude that
\begin{equation}\label{limLambda+}
\lim_{s\uparrow 1} \Lambda^+_s(\alpha) = \lambda_1(\Omega). 
\end{equation}

In order to finish the study of the maximization problem, we need to study the quasi-optimal Dirichlet sets $D$. 
Notice that, since we are considering a maximization problem, it is not clear that there is
an optimal set for $\Lambda_s^+(\alpha)$. Hence, we deal with quasi-optimal sequences as $s\uparrow 1$, that is, we take
$D_s\subset \D_\alpha$ such that
$$
\Lambda_s^+(\alpha) = \lambda_s(D_s) + o(1),
$$
where $o(1)\to 0$ as $s\uparrow 1$. Our next result says that the quasi-optimal sets $D_s$ ``covers'' the boundary of $\Omega$ 
as $s\uparrow 1$.

\begin{thm}
Given any quasi-optimal configuration $D_s$ for $\Lambda_s^+(\alpha)$, any $x\in \partial\Omega$ and $\ve>0$, there exists $s_0\in (0,1)$ such that for every $s_0<s<1$ it holds
$$
|D_s\cap B_\ve(x)|>0.
$$
\end{thm}

\begin{proof}
Assume that the conclusion is false. Then, there exists a sequence $s_k\uparrow 1$, sets $D_k\subset\R^n\setminus \Omega$ and $\ve_0>0$ such that
$$
\lambda_{s_k}(D_k) \ge \Lambda_{s_k}^+(\alpha) - \frac{1}{k}\quad \text{and} \quad |D_k\cap B_{\ve_0}(x)|=0.
$$

Now, by \eqref{limLambda+}, we have
\begin{equation}\label{liminf}
\lambda_1 \le \liminf_{k\to \infty} \lambda_{s_k}(D_k).
\end{equation}

On the other hand, let $\Gamma = \partial\Omega\cap B_{\ve_0}(x)$ 
and consider the following eigenvalue problem (that was already mentioned in the introduction)
\begin{equation}\label{mixed}
\begin{cases}
-\Delta_p u = \lambda |u|^{p-2}u & \text{in }\Omega\\
u= 0 & \text{on } \partial\Omega\setminus \Gamma\\
|\nabla u|^{p-2}\partial_\nu u = 0 & \text{on } \Gamma.
\end{cases}
\end{equation}

Let us denote by $\lambda_1(\Gamma)$ the first eigenvalue of \eqref{mixed}, i.e.
$$
\lambda_1(\Gamma) = \inf \frac{\displaystyle\int_\Omega |\nabla v|^p\, dx}{\displaystyle\int_{\Omega} |v|^p\, dx}
$$
where the infimum is taken over all functions $v\in W^{1,p}(\Omega)$ such that $v=0$ on $\partial\Omega\setminus \Gamma$.

It is straightforward to see that $$0= \lambda_1^N \leq \lambda_1(\Gamma) \leq \lambda_1$$
and that the above inequalities are strict when $\Gamma \neq \emptyset, \partial \Omega$.

Now, denote by $u_1\in W^{1,p}(\Omega)$ the eigenfunction of \eqref{mixed} associated to $\lambda_1(\Gamma)$ extended by 0 on $D_k$ normalized in $L^p(\Omega)$. Then $u_1\in W^{s,p}_{D_k}(\Omega)$ and so
\begin{align*}
\limsup_{k\to\infty} \lambda_{s_k}(D_k) &\le \lim_{k\to\infty} K(n,s_k,p) [u_1]_{s_k, p}^p\\
& = \|\nabla u_1\|_p^p = \lambda_1(\Gamma)\\
& <  \lambda_1.
\end{align*}
This contradicts \eqref{liminf} and the proof is complete.
\end{proof}

\section{The minimization problem}

Now we analize the minimization problem
\begin{equation}\label{inf.l}
\Lambda^-_s(\alpha) := \inf_{D\in \D_\alpha} \lambda_s(D).
\end{equation}
As we mentioned in the introduction, this problem is of little interest since taking a sequence of domains to infinity makes the eigenvalues go to zero.

\begin{prop}
Let $\Omega\subset\R^n$ be bounded and open and let $0<s<1<p<\infty$ be fixed. Then
$$
\Lambda_s^-(\alpha)=0,
$$
for every $\alpha>0$, where $\Lambda^-_s(\alpha)$ is the constant defined in \eqref{inf.l}.
\end{prop}

\begin{proof}
Take $r>0$ such that $|B_r(0)|=\alpha$ and define $D_k= B_r(ke_1)$ where $e_1=(1,0,\dots,0)\in \R^n$ so $D_k\in \D_\alpha$.

Now, let $u_k$ be defined as
$$
u_k(x) = \begin{cases}
|\Omega|^{-\frac{1}{p}} & x\in\Omega\\
0 & x\in D_k.
\end{cases}
$$
Observe that $u_k\in W^{s,p}_{D_k}(\Omega)$ for $k$ large and $\|u_k\|_{p;\Omega} = 1$.

So,
$$
\Lambda_s^-(\alpha)\le \lambda_s(D_k) \le K [u_k]_{s,p;\Omega\cup D_k}^p,
$$
with
$$
[u_k]_{s,p;\Omega\cup D_k}^p = \frac{2}{|\Omega|} \int_{D_k}\int_{\Omega} \frac{1}{|x-y|^{n+sp}}\, dxdy.
$$
Now, just observe that $|x-y|\sim k$ for $x\in\Omega$ and $y\in D_k$ and so there exists a constant $C>0$ independent of $k$ such that
$$
[u_k]_{s,p;\Omega\cup D_k}^p \le C k^{-(n+sp)}.
$$
This completes the proof.
\end{proof}
We consider now the problem $\Lambda_s^{-,R}(\alpha)$ defined by
\begin{equation}\label{inf.ll}
\Lambda_s^{-,R}(\alpha)=\inf_{\D_\alpha^R}\lambda_s(D).
\end{equation}
In this case the admissible sets are forced to be inside  of the ball $B_R(0)$. Note that in this case we have
$\Lambda_s^{-,R}(\alpha)>0$.
Taking the limit as $s\uparrow 1$ in $\Lambda_s^{-,R}(\alpha)$ we obtain that this quantity tends to zero. This is contained in the following proposition.

\begin{prop}\label{-R}
Let $\Omega\subset\R^n$ be bounded and open, let $0<s<1<p<\ito$ be fixed. Then 
$$
\lim_{s\uparrow 1} \Lambda_s^{-,R}(\alpha)=0,
$$
for every $\alpha>0$, where $\Lambda_s^{-,R}(\alpha)$ is the constant defined in \eqref{inf.ll}. 
\end{prop}
\begin{proof}
Let $D\in \D_\alpha^R$ be such that $\dist(D, \Omega)>0$. Now we define the function
$$
u(x) = \begin{cases}
|\Omega|^{-\frac{1}{p}} & x\in\Omega\\
0 & x\in D,
\end{cases}
$$
 so $u\in W_D^{s,p}(\Omega)$ and $\|u\|_{p;\Omega}=1$.

We can then estimate $\Lambda_s^{-,R}(\alpha)$ as follows,
$$
\Lambda_s^{-,R}(\alpha)\le \lambda_s(D)\le K [u]_{s,p,\Omega\cup D}^p.
$$
On the other side
$$
[u]_{s,p,\Omega\cup D}^p = \frac{2}{|\Omega|}\int_{D} \int_{\Omega}\frac{1}{|x-y|^{n+sp}}\,dxdy\le \frac{2\alpha}{\dist(D,\Omega)^{n+sp}},
$$
then 
$$
K[u]_{s,p,\Omega\cap D_R}\le C (1-s),
$$
where the constant $C$ does not depend on $s$. This implies the desired result.
\end{proof}

To end this section, we show that for any fixed $0<s<1$ there exists an optimal configuration for the constant $\Lambda_s^{-,R}(\alpha)$. Let us observe that this does not follows by a direct application of the {\em Direct Method of the Calculus of Variations} since for a minimizing sequence of domains, we do not have enough compactness in a topology of domains and, what is more important, the associated eigenfunctions do not lie in the same functional space.

We begin with the following Lemma.
\begin{lemp}\label{lema1}
Given $u\in W^{s,p}(\Omega)$, there exists $D_u\subset B_R\setminus\Omega$, $|D_u|=\alpha$ such that 
$\phi=\chi_{D_u}$ is a solution to the following minimization problem
$$
\min_\phi \left\{\int_\Omega \int_{B_R\setminus \Omega} \frac{\phi(y) |u(x)|^p}{|x-y|^{n+sp}}\, dy\, dx
\ \colon \ 0\le \phi\le 1, \int_{B_R\setminus \Omega} \phi(y)\, dy=\alpha\right\}.
$$
\end{lemp}

\begin{proof}
The proof is a direct consequence of the {\em Bathtube Principle} \cite[Theorem 1.14]{Lieb-Loss}.
In fact, let 
$$
f(y) := \int_\Omega \frac{|u(x)|^p}{|x-y|^{n+sp}}\, dx,
$$
and observe that the level sets of $f$, $\{f\le t\}$, has finite measure.
So, the Bathtube principle says that the problem
$$
\inf_{\phi}\left\{\int_{B_R\setminus \Omega} \phi(y) f(y)\, dy\colon 0\le \phi\le 1,\ \int_{B_R\setminus \Omega} \phi(y)\, dy=\alpha\right\}
$$
has a solution of the form $\phi = \chi_D$, where $\{f<t\}\subset D\subset \{f\le t\}$, and $|D|=\alpha$ for some $t>0$.

This completes the proof of the Lemma.
\end{proof}

We can now prove the existence of the optimal configuration for $\Lambda^{-,R}_s(\alpha)$.

\begin{thm}\label{existe.config}
Let $\Omega\subset \R^n$ be a bounded domain, let $\alpha>0$ and let $R>0$ be such that $|B_R\setminus \Omega|>\alpha$. Then, there exists an optimal set $D\subset B_R\setminus \Omega$, that is, $D$ such that $|D|=\alpha$ and
$$
\Lambda^{-,R}_s(\alpha) = \lambda_s(D).
$$
\end{thm}

\begin{proof}
Let $D_k\subset B_R\setminus \Omega$ be a minimizing sequence for $\Lambda^{-,R}_s(\alpha)$, that is
$$
|D_k|=\alpha \qquad \text{and}\qquad \lambda_s(D_k)\to \Lambda^{-,R}_s(\alpha).
$$
Let $u_k\in W^{s,p}_{D_k}(\Omega)$ be the normalized, nonnegative eigenfunction associated to $\lambda_s(D_k)$. Then, it is easy to see that the sequence $\{u_k\}_{k\in\N}$ is bounded in $W^{s,p}(\Omega)$ and so, there exists $u\in W^{s,p}(\Omega)$ such that, up to a subsequence,
\begin{align*}
u_k\rightharpoonup u & \text{ weakly in } W^{s,p}(\Omega)\\
u_k\to u & \text{ strongly in } L^p(\Omega) \text{ and a.e. in }\Omega.
\end{align*}

Now, let $D=D_u$ be the optimal set given in Lemma \ref{lema1}. Let us check that this set is optimal for $\Lambda^{-,R}_s(\alpha)$.

In fact, passing to a further subsequence, if necessary, we can assume that there exists $\phi\in L^\infty(B_R\setminus\Omega)$, such that
$$
\chi_{D_k} \stackrel{*}{\rightharpoonup} \phi \quad *-\text{weakly in } L^\infty(B_R\setminus \Omega).
$$
Moreover, this function $\phi$ satisfies
$$
0\le \phi\le 1\quad \text{and}\quad \int_{B_R\setminus\Omega}\phi(y)\, dy=\alpha.
$$

Observe now, that given $x\in \Omega$, the function $y\mapsto |x-y|^{-(n+sp)}$ belongs to $L^1(B_R\setminus\Omega)$. So, for any $x\in \Omega$, we have that
$$
\int_{D_k} \frac{1}{|x-y|^{n+sp}}\, dy\to \int_{B_R\setminus\Omega} \frac{\phi(y)}{|x-y|^{n+sp}}\, dy \quad \text{as }k\to\infty.
$$
We can then apply Fatou's Lemma to conclude that
$$
\int_\Omega\int_{B_R\setminus\Omega} \frac{\phi(y)|u(x)|^p}{|x-y|^{n+sp}}\, dy\, dx \le \liminf_{k\to\infty} \int_\Omega\int_{D_k} \frac{|u_k(x)|^p}{|x-y|^{n+sp}}\, dy\, dx
$$
and so, from our choice of $D$ (c.f. Lemma \ref{lema1}), we obtain
\begin{equation}\label{complemento}
\int_\Omega\int_{D} \frac{|u(x)|^p}{|x-y|^{n+sp}}\, dy\, dx \le \liminf_{k\to\infty} \int_\Omega\int_{D_k} \frac{|u_k(x)|^p}{|x-y|^{n+sp}}\, dy\, dx
\end{equation}
Also, by the weak semicontinuity of the Gagliardo seminorm, we have
\begin{equation}\label{gagliardo}
\iint_{\Omega\times\Omega} \frac{|u(x)-u(y)|^p}{|x-y|^{n+sp}}\, dxdy \le \liminf_{k\to\infty} \iint_{\Omega\times\Omega} \frac{|u_k(x)-u_k(y)|^p}{|x-y|^{n+sp}}\, dxdy
\end{equation}

Observe now that if we extend $u$ to $D$ by zero, we have that
$$
[u]_{s,p;\Omega\cup D}^p = \iint_{\Omega\times\Omega} \frac{|u(x)-u(y)|^p}{|x-y|^{n+sp}}\, dxdy + 2\int_\Omega\int_{D} \frac{|u(x)|^p}{|x-y|^{n+sp}}\, dy\, dx 
$$
and so \eqref{complemento} and \eqref{gagliardo} imply that $u\in W^{s,p}_D(\Omega)$ and that
$$
\lambda_s(D)\le K [u]_{s,p; \Omega\cup D}^p\le \liminf_{k\to\infty} K [u_k]_{s,p; \Omega\cup D_k}^p = \lim_{k\to\infty}\lambda_s(D_k) = \Lambda^{-,R}_s(\alpha).
$$

The proof is complete.
\end{proof}

\section{Proof of Theorem \ref{teo.main2}}

The results for the maximization problem follows exactly as those in Section 3 with the obvious modifications. 

It remains to check the results about the minimization problem.

For this, we first observe that, by Lemma \ref{orden} (with the obvious modifications to include the potential function $V$), it holds that
$$
\lambda_{s,V}(\emptyset) \le \lambda_{s,V}(D),
$$
for every $D\in \D_\alpha$, from where it follows that
\begin{equation}\label{primera.cota}
\lambda_{s,V}(\emptyset)\le \Lambda_{s,V}^-(\alpha)\le  \Lambda_{s,V}^{-,R}(\alpha)
\end{equation}
for any $R>0$.

Recall that $\lambda_{s,V}(\emptyset)$ is the first eigenvalue of the fractional $p-$laplacian plus a potential with homogeneous Neumann boundary conditions and is given by
$$
\lambda_{s,V}(\emptyset) = \inf_{v\in W^{s,p}(\Omega)} \frac{
\displaystyle K[v]_{s,p; \Omega}^p + \int_\Omega V(x) |v|^p\, dx}{\|v\|_{p;\Omega}^p}.
$$

It is not difficult to check (c.f. with Remark \ref{rem.ponce}) that $\lambda_{s,V}(\emptyset)$ converges to the first eigenvalue of the local $p-$laplacian plus a potential with homogeneous Neumann boundary conditions. i.e.
\begin{equation}\label{segunda.cota}
\lim_{s\uparrow 1} \lambda_{s,V}(\emptyset) = \lambda_1^N(V),
\end{equation}
where
$$
\lambda_1^N(V) = \inf_{v\in W^{1,p}(\Omega)} \frac{
\displaystyle \int_\Omega |\nabla v|^p\, dx + \int_{\Omega} V(x) |v|^p\, dx}{\displaystyle \int_\Omega |v|^p\, dx}.
$$

Let us now give an estimate for $\Lambda_{s,V}^{-,R}(\alpha)$. This estimate uses the ideas from Proposition \ref{-R}.

Let $D\in \D_\alpha^R$ be such that $\dist(D,\Omega)=d>0$. So $\Lambda_{s,V}^{-,R}(\alpha)\le \lambda_{s,V}(D)$.

Denote by $u_1\in W^{1,p}(\Omega)$ be the eigenfunction associated to $\lambda_1^N(V)$ normalized such that $\|u_1\|_{p;\Omega}=1$ and extend $u_1$ to $D$ by zero. Since $\dist(D,\Omega)>0$ we have that $u_1\in W^{s,p}_D(\Omega)$ and hence
$$
\lambda_{s,V}(D) \le K[u_1]_{s,p;\Omega\cup D}^p + \int_\Omega V(x) |u_1|^p\, dx.
$$

Now
\begin{align*}
K[u_1]_{s,p;\Omega\cup D}^p &= K[u_1]_{s,p;\Omega}^p + 2K \int_D\int_\Omega \frac{|u_1(x)|^p}{|x-y|^{n+sp}}\, dx\, dy\\
&= I + II. 
\end{align*}
Using the results of \cite{Bourgain-Brezis-Mironescu} it follows that
$$
I = K[u_1]_{s,p;\Omega}^p \to \|\nabla u_1\|_{p;\Omega}^p \quad \text{as } s\uparrow 1.
$$
Finally
$$
II\le 2K \frac{\|u_1\|_{p;\Omega}^p |D|}{d^{n+sp}} = \frac{2\alpha K}{d^{n+sp}}\to 0 \quad \text{as } s\uparrow 1
$$
(recall that $K\sim (1-s)$, from \eqref{defK}). Combining all this we arrive at
\begin{equation}\label{tercera.cota}
\lim_{s\uparrow 1} \lambda_{s,V}(D) \le \|\nabla u_1\|_{p;\Omega}^p + \int_\Omega V(x) |u_1|^p\, dx = \lambda_1^N(V).
\end{equation}

Putting together \eqref{primera.cota}, \eqref{segunda.cota} and \eqref{tercera.cota} we arrive at
$$
\lim_{s\uparrow 1} \Lambda_{s,V}^-(\alpha) = \lim_{s\uparrow 1} \Lambda_{s,V}^{-,R}(\alpha) = \lambda_1^N(V).
$$

To finish this section just observe that the existence of an optimal configuration for the constant $\Lambda_{s,V}^{-,R}(\alpha)$ follows without change as in the proof of Theorem \ref{existe.config}.

\section*{Acknowledgements}

This paper was partially supported by Universidad de Buenos Aires under grant UBACyT 20020130100283BA, by ANPCyT under grant PICT 2012-0153 and by CONICET under grant PIP2015 11220150100032CO. J. Fern\'andez Bonder and J. D. Rossi are members of CONICET.


\bibliographystyle{plain}
\bibliography{bib}

\begin{thebibliography}{10}

\bibitem{Bourgain-Brezis-Mironescu}
Jean Bourgain, Ha{\"i}m Brezis, and Petru Mironescu.
\newblock {Another look at Sobolev spaces}, 2001.
\newblock Original research article appeared at in Optimal Control and Partial
  Differential Equations IOS Press ISBN 1 58603 096 5.

\bibitem{Brasco-Parini-Squassina}
Lorenzo Brasco, Enea Parini, and Marco Squassina.
\newblock Stability of variational eigenvalues for the fractional
  {$p$}-{L}aplacian.
\newblock {\em Discrete Contin. Dyn. Syst.}, 36(4):1813--1845, 2016.

\bibitem{Bucur2}
Dorin Bucur and Giuseppe Buttazzo.
\newblock {\em Variational methods in some shape optimization problems}.
\newblock Appunti dei Corsi Tenuti da Docenti della Scuola. [Notes of Courses
  Given by Teachers at the School]. Scuola Normale Superiore, Pisa, 2002.

\bibitem{Bucur1}
Dorin Bucur and Paola Trebeschi.
\newblock Shape optimisation problems governed by nonlinear state equations.
\newblock {\em Proc. Roy. Soc. Edinburgh Sect. A}, 128(5):945--963, 1998.

\bibitem{Ch}
S.~Chanillo, D.~Grieser, M.~Imai, K.~Kurata, and I.~Ohnishi.
\newblock Symmetry breaking and other phenomena in the optimization of
  eigenvalues for composite membranes.
\newblock {\em Comm. Math. Phys.}, 214(2):315--337, 2000.

\bibitem{DiNezza-Palatucci-Valdinoci}
Eleonora Di~Nezza, Giampiero Palatucci, and Enrico Valdinoci.
\newblock Hitchhiker's guide to the fractional {S}obolev spaces.
\newblock {\em Bull. Sci. Math.}, 136(5):521--573, 2012.

\bibitem{Bonder-Spedaletti}
J.~{Fernandez Bonder} and J.~{Spedaletti}.
\newblock Some nonlocal optimal design problems.
\newblock {\em ArXiv e-prints}, January 2016.

\bibitem{henrot1}
Antoine Henrot and Edouard Oudet.
\newblock Minimizing the second eigenvalue of the {L}aplace operator with
  {D}irichlet boundary conditions.
\newblock {\em Arch. Ration. Mech. Anal.}, 169(1):73--87, 2003.

\bibitem{Lieb-Loss}
Elliott~H. Lieb and Michael Loss.
\newblock {\em Analysis}, volume~14 of {\em Graduate Studies in Mathematics}.
\newblock American Mathematical Society, Providence, RI, second edition, 2001.

\bibitem{Ponce}
Augusto~C. Ponce.
\newblock A new approach to {S}obolev spaces and connections to
  {$\Gamma$}-convergence.
\newblock {\em Calc. Var. Partial Differential Equations}, 19(3):229--255,
  2004.

\bibitem{S1}
Jan Soko\l~owski and Jean-Paul Zol\'esio.
\newblock {\em Introduction to shape optimization}, volume~16 of {\em Springer
  Series in Computational Mathematics}.
\newblock Springer-Verlag, Berlin, 1992.
\newblock Shape sensitivity analysis.

\bibitem{S2}
V.~\v~Sver\'ak.
\newblock On optimal shape design.
\newblock {\em J. Math. Pures Appl. (9)}, 72(6):537--551, 1993.

\end{thebibliography}

\end{document}